\documentclass[12pt]{article}
\usepackage{bbm,amsmath}
\usepackage{amssymb}
\usepackage{amsthm}
\usepackage{amsfonts}
\usepackage{latexsym}
\usepackage{color}

\newtheorem{thm}{Theorem}[section]
\newtheorem{defn}[thm]{Definition}
\newtheorem{pro}[thm]{Proposition}
\newtheorem{lem}[thm]{Lemma}

\newtheorem{case}{Case}
\def\cl{\centerline}
\def\vs{\vspace*}
\newcommand{\C}{\mathbb{C}}

\newcommand{\Z}{\mathbb{Z}}

\def\pa{\partial}

\numberwithin{equation}{section}
\allowdisplaybreaks

\begin{document}
\begin{center}
{\bf The Lie conformal algebra of a Block type Lie algebra}
\footnote {Supported by NSF grant no. 10825101 , 11001200 and 11071187 of China and the Fundamental Research Funds for the Central Universities.

Corresponding author: Xiaoqing Yue(xiaoqingyue@tongji.edu.cn).}
\end{center}

\cl{{Ming Gao$^{\dag}$},\ \ \ {Ying Xu$^{\dag}$},\ \ \ Xiaoqing Yue$^{\,\ddag}$}

\cl{\small
$^{\dag\,}$Wu Wen-Tsun Key Laboratory of \vs{-2pt}Mathematics}
\cl{\small  University of Science and
Technology of China, Hefei 230026, China}
\cl{\small $^{\ddag\,}$Department of Mathematics, Tongji
University, Shanghai 200092, China}

\cl{\small E-mail: nina@mail.ustc.edu.cn, xying@mail.ustc.edu.cn, xiaoqingyue@tongji.edu.cn}\vs{5pt}

{\small
\parskip .005 truein
\baselineskip 3pt \lineskip 3pt \noindent{\bf Abstract.}
Let $L$ be a Lie algebra of Block type over $\C$ with basis $\{L_{\alpha,i}\,|\,\alpha,i\in\Z\}$ and brackets $[L_{\alpha,i},L_{\beta,j}]=(\beta(i+1)-\alpha(j+1))L_{\alpha+\beta,i+j}$. In this paper, we shall construct a formal distribution Lie algebra of $L$. Then we decide its conformal algebra $B$ with $\C[\partial]$-basis $\{L_\alpha(w)\,|\,\alpha\in\Z\}$ and $\lambda$-brackets $[L_\alpha(w)_\lambda L_\beta(w)]=(\alpha\partial+(\alpha+\beta)\lambda)L_{\alpha+\beta}(w)$. Finally, we give a classification of free intermediate series $B$-modules. \vs{5pt}


}
\parskip .001 truein\baselineskip 6pt \lineskip 6pt
\vs{10pt}

\section{Introduction}
Lie conformal algebra encodes an axiomatic description of the operator product expansion of chiral fields in conformal field theory. Kac introduced the notion of the conformal algebra in \cite{K1, K2}. It turns out to be an adequate tool for the study of some infinite dimensional Lie algebras. Fattori and Kac gave a complete classification of finite simple Lie conformal superalgebras in \cite{FK}. And the cohomology theory of finite simple Lie conformal superalgebras was developed in \cite{BDK, BKV}.

In this paper, we would like to construct and study a new Lie conformal algebra related to Block type Lie algebras. Block type Lie algebras were first introduced by Block in \cite{B}. Let $L$ be a Lie algebra over $\C$ with basis $\{L_{\alpha,i}\,|\,\alpha,i\in\Z\}$ and Lie brackets defined by
\begin{equation}\label{brackets}
[L_{\alpha,i},L_{\beta,j}]=(\beta(i+1)-\alpha(j+1))L_{\alpha+\beta,i+j}.
\end{equation}
Then $L$ is called {\it a Lie algebra of Block type}, whose general definitions were given in \cite{DZ}, which also comes up as a generalization of the Lie algebra studied in \cite{Su,Su2}. Block type Lie algebras have been widely studied in mathematical literatures(\cite{DZ},\,\cite{OZ}--\cite{X},\,\cite{ZZ}--\cite{ZM}). However, there is little about its conformal algebra. We believe this article would play
an energetic role on the study of Block type Lie algebras.

This paper proceeds as follows. In section 2, the preliminaries of conformal algebra are recalled. In Section 3, we will start from $L$ to construct its Lie algebra of formal distributions. Then we shall construct the related Lie conformal algebra $B$, which is our main interest of this paper. Finally, we will go on to study the representations of $B$ and give a classification of its free intermediate series modules in Section 4.

The main results in this article are the following two theorems.

\begin{thm}\label{firstthm}
Let $B$ be a free $\C[\partial]$-module with basis $\{L_\alpha(w)\,|\,\alpha\in\Z\}$, that is, $B=\bigoplus_{\alpha\in\Z}\C[\partial]L_\alpha(w)$. We define a $\lambda$-bracket on the basis of $B$ as
\begin{equation*}
[L_\alpha(w)_\lambda L_\beta(w)]=(\alpha\partial+(\alpha+\beta)\lambda)L_{\alpha+\beta}(w),
\end{equation*}
and expand by relations $(\ref{lambda1})-(\ref{lambda4})$. Then $B$ becomes a Lie conformal algebra.
\end{thm}

\begin{thm}\label{thm-mod}
Any nontrivial free intermediate series $B$-module must be isomorphic to one of the following three classes$:$\vspace{12pt}

\noindent $V_{C,D}=\bigoplus\limits_{\gamma\in\Z}\C[\partial]v_\gamma:$
  \[L_\alpha(w)_\lambda v_\gamma=((\alpha+\gamma+C)\lambda+\alpha(\partial+D))v_{\alpha+\gamma};\]
$V_D=\bigoplus\limits_{\gamma\in\Z}\C[\partial]v_\gamma:$
\[L_\alpha(w)_\lambda v_\gamma=\left\{ \begin{array}{ll}
((\alpha+\gamma)\lambda+\alpha(\partial+D))v_{\alpha+\gamma}&\mbox{if $\gamma\neq 0$ and $\alpha+\gamma\neq 0$,}\\[4pt]
\alpha (\lambda+\partial+D)^2v_{\alpha+\gamma}&\mbox{if $\gamma= 0$,}\\[4pt]
\alpha\, v_{\alpha+\gamma} &\mbox{if $\alpha+\gamma= 0;$}
 \end{array}  \right.
\]
$V'_D=\bigoplus\limits_{\gamma\in\Z}\C[\partial]v_\gamma:$
\[
L_\alpha(w)_\lambda v_\gamma=\left\{ \begin{array}{ll}
((\alpha+\gamma)\lambda+\alpha(\partial+D))v_{\alpha+\gamma}& \mbox{if $\gamma\neq 0$ and $\alpha+\gamma\neq 0$,}\\[4pt]
\alpha\, v_{\alpha+\gamma} & \mbox{if $\gamma= 0$,}\\[4pt]
\alpha (\partial+D)^2v_{\alpha+\gamma} & \mbox{if $\alpha+\gamma= 0$,}
 \end{array}  \right.
\]
where the module actions are defined on basis elements and $C,D\in\C$ are constants.
\end{thm}

Note that $V_{C,D}$ is isomorphic to $V_{C+n,D}$ for any $n\in\Z$.

\section{Preliminaries}
First we shall list some definitions and results related to formal distributions and $\lambda$-products.

\begin{defn}\label{defi-value}
Let $g$ be a Lie algebra. A formal distribution is called {\bf $g$-valued} if its coefficients are all in $g$.
\end{defn}
We define the {\bf Dirac's $\delta$ distribution} to be
\begin{equation*}
\delta(z,w)=\sum\limits_{n\in\Z}z^{-1}(\frac{w}{z})^n=\sum\limits_{n\in\Z}w^{-1}(\frac{z}{w})^n.
\end{equation*}
We already have the following lemma about Dirac's $\delta$ distribution.
\begin{lem}\label{lem0}
If $m,n\in\mathbb{N}$ and $m>n$, then $(z-w)^m\partial_w^n\delta(z,w)=0$.
\end{lem}
A formal distribution $a(z,w)$ is called {\bf local} if there exists an $N\in\Z_+$ such that $(z-w)^Na(z,w)=0.$ The following proposition describes an equivalent condition for a formal distribution to be local with the help of Dirac's $\delta$:

\begin{pro}\label{intr1}
A formal distribution $a(z,w)$ is local if and only if it is an expansion of $\partial^j_w\delta(z,w),\,j\in\Z$. That is, $a(z,w)$ can be written into
\begin{equation*}
a(z,w)=\sum\limits_{j\in\Z^+}c^j(w)\frac{\partial^j_w \delta(z,w)}{j!},
\end{equation*}
where $c^j(w)\in g[[w,w^{-1}]]$.
\end{pro}

\begin{defn}\label{defi-j}
Let $a(w)$ and $b(w)$ be formal distributions. Their {\bf $j$-product} and {\bf $\lambda$-product} are defined by
\begin{equation*}
a(w)_{(j)}b(w)= Res_z (z-w)^j [a(z),b(w)]
\end{equation*}
and
\begin{equation*}
[a(w)_\lambda b(w)]=\Phi^\lambda_{z,w}[a(z),b(w)],
\end{equation*}
where $\Phi^\lambda_{z,w}=Res_ze^{\lambda(z-w)}$.
\end{defn}

For local formal distributions, we have
\begin{pro}\label{prop3} If $a(z,w)$ is local, then 
\begin{equation*}
\Phi^\lambda_{z,w}\partial_z=-\lambda\Phi^\lambda_{z,w},
\end{equation*}
\begin{equation*}
\Phi^\lambda_{z,w}a(z,w)=\Phi^{-\lambda-\partial_w}_{z,w}a(z,w).
\end{equation*}
\end{pro}

The $j$-products and $\lambda$-products are just the two sides of the same coin:
\begin{pro}\label{j-lambda}
The $j$-products and $\lambda$-products are related by the following identity
\begin{equation*}
[a_\lambda b]=\sum_{j\in\Z^+}\frac{\lambda^j}{j!}(a_{(j)}b).
\end{equation*}
\end{pro}

Now, we introduce some definitions about Lie conformal algebra.
\begin{defn}\label{conformal algebra}
A {\bf Lie conformal algebra} $R$ is a left $\C[\partial]$-module endowed with a $\lambda$-bracket $[a_{\lambda}b]$ which defines a linear map $R\otimes R\rightarrow R[\lambda]=\C[\lambda]\otimes R$, subject to the following three axioms$:$
\begin{eqnarray}
&&[\partial a_\lambda b]=-\lambda[a_\lambda b],\ \ \  [a_\lambda \partial b]=(\partial+\lambda)[a_\lambda b],\ \ { (conformal\  sesquilinearity)},\label{lambda1}\label{lambda2}\\
&&[b_\lambda a]=-[a_{-\partial-\lambda} b],\ \ \ \ \ \ \ \ \ \ \ \ \ \ \ \ \ \ \ \ \ \ \ \ \ \ \ \ \ \ \ {(skew\ symmetry)},\label{lambda3}\\
&&[[a_\lambda b]_{\lambda+\mu}c]=[a_\lambda [b_\mu c]]-[b_\mu[a_\lambda c]],\ \ \ \ \ \ \ \ \ \ \ \ \ \ {(Jacobi\ identity)}.\label{lambda4}
\end{eqnarray}
A Lie conformal algebra $R$ is {\bf finite} if it is finitely generated as a $\C[\pa]$-module.
\end{defn}

Suppose there is a local family $F$ of $g$-valued formal distributions whose coefficients can generate the whole $g$, then $F$ is called a formal distribution Lie algebra of $g$. We denote by $\bar{F}$ the minimal subspace of $g[[z,z^{-1}]]$, which is closed under the
$j$-products and contains $F$. Then $\bar{F}$ is called the minimal local family of $F$. Let $R$ be a subset of $g[[z,z^{-1}]]$, which is closed under the derivation $\partial$ over $z$. If $R$ contains $\bar{F}$, we shall call it a conformal family. Note that a conformal family $R$ has the structure of a $\C[\partial]$-module.

\begin{defn}\label{confor module}
A $R$-{\rm\bf module} $V$ is defined to be a $\C[\partial]$-module with a $\lambda$-action
\begin{center}
$a_\lambda v: R\times V\rightarrow V[[\lambda]]$
\end{center} such that for any $a,b\in R,v\in V$, we have

\begin{eqnarray}
\label{mod}
&&[a_\lambda b]_{\lambda+\mu}v=a_\lambda(b_\mu v)-b_\mu(a_\lambda v),\\
\label{mod1}
&&(\partial a)_\lambda v=-\lambda a_\lambda v,\\
&&\label{mod2}
a_\lambda(\partial v)=(\partial+\lambda)a_\lambda v.
\end{eqnarray}
We call a $R$-module $V$ {\bf conformal} if $a_\lambda v\in V[\lambda]$ for all $a\in R,v\in V;$ {\bf $\Z$-graded} if $V=\bigoplus_{\gamma\in\Z}V_\gamma$ as a $\C[\partial]$-module and $(R_\alpha)_\lambda V_\gamma\subset V_{\alpha+\gamma}$ for any $\alpha,\gamma\in\Z$. In addition$,$ if each $V_\gamma$ can be generated by one element $v_\gamma\in V_\gamma$ over $\C[\partial],$ we call $V$ an {\bf intermediate series} $R$-module. An intermediate series $R$-module $V$ is called {\bf free} if each $V_\gamma$ is freely generated by some $v_\gamma\in V_\gamma$ over $\C[\partial]$.
\end{defn}

\section{The Lie Conformal Algebra $B$}
In this section, we shall start from the Lie algebra $L$ to construct a Lie conformal algebra $B$ via formal distribution Lie algebras.

Let $L_\alpha(z)=\sum_{i\in\Z}L_{\alpha,i-1}z^{-i-1}\in B[[z,z^{-1}]]$, for any $\alpha\in\Z$. Then $F=\{L_\alpha(z)\,|\,\alpha\in\Z\}$ is a set of $L$-valued formal distributions.
We can use $F$ to express the Lie brackets of $L$ into formal distributions as:
\begin{pro}\label{prop0}
The commutation relation between $L_\alpha(z)$ and $L_\beta(w)$ is
\begin{equation*}\label{str1}
[L_\alpha(z),L_\beta(w)]=\alpha\partial_wL_{\alpha+\beta}(w)\delta(z,w)+(\alpha+\beta)L_{\alpha+\beta}(w)\partial_w\delta(z,w).
\end{equation*}
\end{pro}
\begin{proof}
Using equation (\ref{brackets}), we obtain
\begin{align*}
[L_\alpha(z),L_\beta(w)]&=[\sum_{i\in\Z}L_{\alpha,i-1}z^{-i-1},\sum_{j\in\Z}L_{\beta,j-1}w^{-j-1}]\\
                        &=\sum_{i,j\in\Z}[L_{\alpha,i-1}L_{\beta,j-1}]z^{-i-1}w^{-j-1}\\
                        &=\sum_{i,j\in\Z}(\beta i-\alpha j)L_{\alpha+\beta,i+j-2}z^{-i-1}w^{-j-1}\\
                        &=-\alpha\sum_{i,j\in\Z}jL_{\alpha+\beta,i+j-2}z^{-i-1}w^{-j-1}+\beta\sum_{i,j\in\Z}iL_{\alpha+\beta,i+j-2}z^{-i-1}w^{-j-1}\\
                        &=\alpha\partial_w\sum_{i,j\in\Z}L_{\alpha+\beta,i+j-2}z^{-i-1}w^{-j}+\beta\sum_{i,k\in\Z}iL_{\alpha+\beta,k-1}z^{-i-1}w^{-k+i-2}\\
                        &=\alpha\partial_w(\sum_{k\in\Z}L_{\alpha+\beta,k-1}w^{-k-1}\sum_{i\in\Z}z^{-i-1}w^i)\\
                        &\ \ \ \ +\beta\sum_{k\in\Z}L_{\alpha+\beta,k-1}w^{-k-1}\sum_{i\in\Z}iz^{-i-1}w^{i-1}\\
                        &=\alpha\partial_w(L_{\alpha+\beta}(w)\delta(z,w))+\beta L_{\alpha+\beta}(w)\partial_w\delta(z,w)\\
                        &=\alpha\partial_wL_{\alpha+\beta}(w)\delta(z,w)+(\alpha+\beta)L_{\alpha+\beta}(w)\partial_w\delta(z,w),
\end{align*}
where $k=i+j-1$.
\end{proof}

By Proposition \ref{intr1}, we know that $[L_\alpha(z),L_\beta(w)]$ is local for any $\alpha,\beta\in\Z$, which suggests $F$ is a local family of formal distributions. Since the coefficients of $F$ is a basis of $L$, we conclude that $F$ is a formal distribution Lie algebra of $L$.

\begin{pro}\label{prop1}
In terms of $\lambda$ brackets, we have
\begin{equation*}
[L_\alpha(w)_\lambda L_\beta(w)]=(\alpha\partial+(\alpha+\beta)\lambda)L_{\alpha+\beta}(w).
\end{equation*}
\end{pro}

\begin{proof}
By Definition \ref{defi-j} and Proposition \ref{prop0}, we have
\begin{align*}
[L_\alpha(w)_\lambda L_\beta(w)]&=Res_ze^{\lambda(z-w)}[L_\alpha(z)L_\beta(w)]\\
                                &=Res_ze^{\lambda(z-w)}(\alpha\partial_wL_{\alpha+\beta}(w)\delta(z,w)+(\alpha+\beta)L_{\alpha+\beta}(w)\partial_w\delta(z,w)).
\end{align*}
The first term can be calculated as
\begin{align*}
Res_ze^{\lambda(z-w)}\partial_wL_{\alpha+\beta}(w)\delta(z,w)&=Res_z\sum_{i=0}^\infty\frac{\lambda^i}{i!}(z-w)^i\partial_wL_{\alpha+\beta}(w)\delta(z,w)\\
                                                             &=\partial_wL_{\alpha+\beta}(w)\sum_{i=0}^\infty\frac{\lambda^i}{i!}Res_z(z-w)^i\delta(z,w),
\end{align*}
Since by Proposition \ref{lem0}, we know that $Res_z(z-w)^i\delta(z,w)=0$ for $i\geq 1$, we have
\begin{center}
    $Res_ze^{\lambda(z-w)}\partial_wL_{\alpha+\beta}(w)\delta(z,w)=\partial_wL_{\alpha+\beta}(w)Res_z\delta(z,w)=\partial_wL_{\alpha+\beta}(w)$.
\end{center}
Similarly, for the second term
\begin{align*}
Res_ze^{\lambda(z-w)}L_{\alpha+\beta}(w)\partial_w\delta(z,w)&=Res_z\sum_{i=0}^\infty\frac{\lambda^i}{i!}(z-w)^i L_{\alpha+\beta}(w)\partial_w\delta(z,w)\\
                                                             &=L_{\alpha+\beta}(w)\sum_{i=0}^\infty\frac{\lambda^i}{i!}Res_z (z-w)^i\partial_w\delta(z,w),
\end{align*}
by Proposition \ref{lem0} again, we have $Res_z(z-w)^i\partial_w\delta(z,w)=0$ for all $i\geq 2$, which means 
\begin{align*}
&Res_ze^{\lambda(z-w)}L_{\alpha+\beta}(w)\partial_w\delta(z,w)\\
&=L_{\alpha+\beta}(w)(Res_z\partial_w\delta(z,w)+\lambda Res_z(z-w)\partial_w\delta(z,w))\\
&=\lambda L_{\alpha+\beta}(w).
\end{align*}
Hence we obtain
\begin{equation*}
[L_\alpha(w)_\lambda L_\beta(w)]=\alpha\partial_w L_{\alpha+\beta}(w)+(\alpha+\beta)\lambda L_{\alpha+\beta}(w).
\end{equation*}
\end{proof}

The equality in Proposition \ref{prop1} can be rewritten as:
\begin{center}
$[L_\alpha(w)_\lambda L_\beta(w)]=\frac{\lambda^0}{0!}\alpha\partial L_{\alpha+\beta}(w) +\frac{\lambda^1}{1!}(\alpha+\beta)L_{\alpha+\beta}(w).$
\end{center}

By Proposition \ref{j-lambda}, this implies

\begin{pro}\label{prop2}
The $j$-products corresponding to the $\lambda$-products of $F$ are
$L_\alpha(w)_{(0)} L_\beta(w)=\alpha\partial L_{\alpha+\beta}(w)$, $L_\alpha(w)_{(1)} L_\beta(w)=(\alpha+\beta)L_{\alpha+\beta}(w)$ and $L_\alpha(w)_{(j)} L_\beta(w)=0$ for $j\neq 0,1$.
\end{pro}

Therefore, $\C[\partial]F=\bigoplus_{\alpha\in\Z}\C[\partial]L_\alpha$ is closed under $j$-products and we have $\bar{F}=\bigoplus_{\alpha\in\Z}\C[\partial]L_\alpha$. Since $\C[\partial]F$ has a natural $\C[\partial]$-module structure, we obtain the corresponding conformal family $B=\C[\partial]F=\bigoplus_{\alpha\in\Z}\C[\partial]L_\alpha$.

Now we claim $B$ is an infinite dimensional Lie conformal algebra. We shall check the conformal algebra structure on $B$ directly as follows:

\begin{proof}[Proof of Theorem \ref{firstthm}]
We need to check $(\ref{lambda1})-(\ref{lambda4})$ for the basis elements of $B$.

It is easy to see that (\ref{lambda1}) can be proved immediately from Proposition \ref{prop3}, and (\ref{lambda3}) can be checked as follows,
\begin{eqnarray*}
[L_\beta(w)_\lambda L_\alpha(w)]\!\!\!&=&\!\!\!(\beta\partial+(\alpha+\beta)\lambda)L_{\alpha+\beta}(w)\\
\!\!\!&=&\!\!\!-(\alpha\partial+(\alpha+\beta)(-\lambda-\partial))L_{\alpha+\beta}(w)\\
\!\!\!&=&\!\!\!-[L_\alpha(w)_{-\lambda-\partial}L_\beta(w)].
\end{eqnarray*}

Now we compute the terms of Jacobi identity (\ref{lambda4}) separately.
\begin{align}\label{equation1}
\begin{split}
    [L_\alpha(w)_\lambda[L_\beta(w)_\mu L_\gamma(w)]]&=[L_\alpha(w)_\lambda(\beta\partial+(\beta+\gamma)\mu)L_{\beta+\gamma}(w)]\\
                                                     &=\beta[L_\alpha(w)_\lambda\partial L_{\beta+\gamma}(w)]+\mu(\beta+\gamma)[L_\alpha(w)_\lambda L_{\beta+\gamma}(w)]\\
                                                     &=\beta(\partial+\lambda)(\alpha\partial+(\alpha+\beta+\gamma)\lambda)L_{\alpha+\beta+\gamma}(w)\\
                                                     &\quad+\mu(\beta+\gamma)(\alpha\partial+(\alpha+\beta+\gamma)\lambda)L_{\alpha+\beta+\gamma}(w)\\
                                                     &=(\beta\partial+\beta\lambda+\mu\beta+\mu\gamma)
                                                     (\alpha\partial+(\alpha+\beta+\gamma)\lambda)L_{\alpha+\beta+\gamma}(w).
\end{split}
\end{align}
Similarly, we have
\begin{align}\label{equation2}
\begin{split}
[[L_\alpha(w)_\lambda L_\beta(w)]_{\lambda+\mu}L_\gamma(w)]&=[(\alpha\partial+(\alpha+\beta)\lambda)L_{\alpha+\beta}(w)_{\lambda+\mu}L_\gamma(w)]\\
                                                           &=-(\lambda+\mu)\alpha((\alpha+\beta)\partial+(\alpha+\beta+\gamma)(\lambda+\mu))L_{\alpha+\beta+\gamma}(w)\\
                                                           &\quad +(\alpha+\beta)\lambda((\alpha+\beta)\partial+(\alpha+\beta+\gamma)(\lambda+\mu))L_{\alpha+\beta+\gamma}(w)\\
                                                           &=(\lambda\beta-\mu\alpha)((\alpha+\beta)\partial+(\alpha+\beta+\gamma)(\lambda+\mu))L_{\alpha+\beta+\gamma}(w),
\end{split}
\end{align}
and
\begin{equation}\label{equation3}
    [L_\beta(w)_\mu[L_\alpha(w)_\lambda L_\gamma(w)]]=(\alpha\partial+\alpha\mu+\alpha\lambda+\gamma\lambda)(\beta\partial+(\alpha+\beta+\gamma)\mu)L_{\alpha+\beta+\gamma}(w).
\end{equation}
Note that (\ref{equation3}) can be obtained by subtracting (\ref{equation2}) from (\ref{equation1}), that is,
\begin{equation*}
    [L_\beta(w)_\mu[L_\alpha(w)_\lambda L_\gamma(w)]]=[L_\alpha(w)_\lambda[L_\beta(w)_\mu L_\gamma(w)]-[[L_\alpha(w)_\lambda L_\beta(w)]_{\lambda+\mu}L_\gamma(w)],
\end{equation*}
which proves the Jacobi identity.
\end{proof}
Note that $B$ is a $\mathbb{Z}$-graded Lie conformal algebra in the sense $B=\bigoplus_{\alpha\in\Z}B_\alpha$, where $B_\alpha=\C[\partial]L_\alpha(w)$.

\section{The Representations of B}\setcounter{case}{0}

In this section, we will give a classification of all free intermediate series modules over $B$.

Let $V$ be an arbitrary free intermediate series $B$-module. Then as a $\C[\partial]$-module, $V=\bigoplus_{\gamma\in\Z} V_\gamma$, where each $V_\gamma$ is freely generated by some element $v_\gamma\in V_\gamma$. For any $\alpha,\gamma\in\Z$, denote $L_\alpha(w)_\lambda v_\gamma=f_{\alpha,\gamma}(\lambda,\partial)v_{\alpha+\gamma}$. We call $\{f_{\alpha,\gamma}(\lambda,\partial)\in \C[\lambda,\partial]\,|\,\alpha,\gamma\in\Z\}$ the structure coefficients of $V$ associated with the $\C[\partial]$-basis $\{v_\gamma\,|\,\gamma\in\Z\}$. We can see that $V$ is determined if and only if all of its structure coefficients are specified.

\begin{pro}
The structure coefficients $f_{\alpha,\gamma}(\lambda,\partial)$ of a free intermediate series $B$-module V must satisfy the following equality
\begin{equation}\label{use}
(\beta\lambda-\alpha\mu)f_{\alpha+\beta,\gamma}(\lambda+\mu,\partial)
=f_{\beta,\gamma}(\mu,\partial+\lambda)f_{\alpha,\beta+\gamma}(\lambda,\partial)-f_{\alpha,\gamma}(\lambda,\partial+\mu)f_{\beta,\alpha+\gamma}(\mu,\partial).
\end{equation}
\end{pro}
\begin{proof}
By Definition \ref{confor module}, we can write equation (\ref{mod}) with respect to the basis elements of $B$ and $V$ as
\begin{equation}\label{lala}
[L_\alpha(w)_\lambda L_\beta(w)]_{\lambda+\mu}v_\gamma=
L_\alpha(w)_\lambda(L_\beta(w)_\mu v_\gamma)-
L_\beta(w)_\mu(L_\alpha(w)_\lambda v_\gamma),
\end{equation}
for any $\alpha,\beta,\gamma\in\mathbb{Z}$.

The left side of the equation can be calculated as
\begin{align*}
&[L_\alpha(w)_\lambda L_\beta(w)]_{\lambda+\mu}v_\gamma\\
=&(((\alpha+\beta)\lambda+\alpha\partial)L_{\alpha+\beta}(w))_{\lambda+\mu}v_\gamma\\
=&((\alpha+\beta)\lambda-\alpha(\lambda+\mu))L_{\alpha+\beta}(w)_{\lambda+\mu}v_\gamma\\
=&(\beta\lambda-\alpha\mu)f_{\alpha+\beta,\gamma}(\lambda+\mu,\partial)v_{\alpha+\beta+\gamma}.
\end{align*}
Similarly
$$L_\alpha(w)_\lambda(L_\beta(w)_\mu v_\gamma)=f_{\beta,\gamma}(\mu,\partial+\lambda)f_{\alpha,\beta+\gamma}(\lambda,\partial)v_{\alpha+\beta+\gamma},$$
$$
L_\beta(w)_\mu(L_\alpha(w)_\lambda v_\gamma)=f_{\alpha,\gamma}(\lambda,\partial+\mu)f_{\beta,\alpha+\gamma}(\mu,\partial)v_{\alpha+\beta+\gamma}.$$
Hence follows the equality (\ref{use}).
\end{proof}

From now on, we shall focus on analyzing the structure coefficients of an intermediate series $B$-module $V$ using equation (\ref{use}). We will sometimes use the notation $f_{x,y}$ instead of $f_{x,y}(\lambda,\partial)$ for convenience.

\begin{lem}\label{pre4}
\begin{description}
  \item[(a)] If $f_{1,\gamma_0}=0$ for some $\gamma_0$, then $f_{\alpha,\gamma}=0$ for $\gamma\leq\gamma_0$ and $\alpha+\gamma\geq\gamma_0+1$.
  \item[(b)] If $f_{-1,\gamma_0}=0$ for some $\gamma_0$, then $f_{\alpha,\gamma}=0$ for $\gamma\geq\gamma_0$ and $\alpha+\gamma\leq\gamma_0+1$.
\end{description}
\end{lem}
\begin{proof}
We prove (a) first.

Let's use induction on $\gamma$. If $\gamma=\gamma_0$, let $\beta=\alpha=1,\gamma=\gamma_0$ in equation \eqref{use}, we have $f_{2,\gamma_0}=0$, by induction on $\alpha$, the result follows.

Suppose the result holds for some $\gamma=n\leq\gamma_0$, that is, $f_{\alpha,n}=0$ for all $\alpha\geq\gamma_0-\gamma+1=\gamma_0-n+1$. If $\gamma=n-1$, we need to prove $f_{\alpha,n-1}=0$ for all $\alpha\geq\gamma_0-n+2$.

Let $\gamma=n-1,\alpha=\gamma_0-n+1,\beta=1$ in equation \eqref{use}, we have $f_{\gamma_0-n+2,n-1}=0$. Suppose $f_{m,n-1}=0$ for some $m\geq\gamma_0-n+2$. Since $f_{m,n}=0$ by the supposition on $\gamma$, using equation \eqref{use}, we have $f_{m+1,n-1}=0$. Thus by induction on $\alpha$, we obtain the lemma.

By symmetry, (b) can be proved similarly using the steps of (a).
\end{proof}

\begin{lem}\label{strive1}
\begin{description}
\item[(a)] If $f_{1,\gamma_1}=f_{1,\gamma_2}= 0$ for some $\gamma_1<\gamma_2$, then $f_{1,\gamma}=0$ for $\gamma_1\leq\gamma\leq\gamma_2$.
\item[(b)] If $f_{-1,\gamma_1}=f_{-1,\gamma_2}=0$ for some $\gamma_1<\gamma_2$, then $f_{-1,\gamma}=0$ for $\gamma_1\leq\gamma\leq\gamma_2$.
\end{description}
\end{lem}
\begin{proof}
We need only to prove (a).
Let $\alpha=\gamma_2-\gamma_1+1$ and $\beta=\gamma_1-\gamma_2$. Since $\alpha+\gamma=\gamma_2+\gamma-\gamma_1+1\geq\gamma_2+1$ and $\gamma\leq\gamma_2$, by Lemma \ref{pre4} (a), $f_{1,\gamma_2}= 0$ gives us $f_{\alpha,\gamma}=0$. Similarly since $\alpha+\beta+\gamma=\gamma+1\geq\gamma_1+1$ and $\beta+\gamma=\gamma_1+\gamma-\gamma_2\leq\gamma_1$, $f_{1,\gamma_1}=0$ implies $f_{\alpha,\beta+\gamma}=0$. Consequently from equation \eqref{use}, we know that $f_{1,\gamma}=f_{\alpha+\beta,\gamma}=0$.
\end{proof}

Roughly speaking, the structure of an intermediate series $B$-module $V=\bigoplus_{\gamma\in\Z}V_\gamma=\bigoplus_{\gamma\in\Z}\C[\partial]v_\gamma$ must belong to one of the following two cases:

\begin{case}
The truncated-submodule case
\end{case}

If there is a $\gamma_0\in\Z$, such that $f_{1,\gamma_0}=f_{-1,\gamma_0+1}= 0$, we call $V$ an intermediate $B$ module with truncated submodules.

By Lemma \ref{strive1}, there must exist some $\gamma_1,\gamma_2\in\Z$, such that $f_{1,\gamma_1-1}\neq 0$, $f_{1,\gamma_2+1}\neq 0$ and $f_{1,\gamma}=0$ for $\gamma_1\leq\gamma\leq\gamma_2$. Similarly, there are $\gamma_3,\gamma_4\in\Z$, such that $f_{1,\gamma_3-1}\neq 0$, $f_{1,\gamma_4+1}\neq 0$ and $f_{-1,\gamma}= 0$ for $\gamma_3\leq\gamma\leq\gamma_4$. Let $p=\max\{\gamma_1,\gamma_3-1\},q=\min\{\gamma_2+1,\gamma_4\}$. Then by Lemma \ref{pre4}, $W_1=\bigoplus_{\gamma\leq p}V_\gamma$ and $W_2=\bigoplus_{\gamma\geq q}V_\gamma$ are both $B$-submodules of $V$. We call $W_1,W_2$ truncated intermediate series $B$-modules. Hence as a $B$-module, $V$ is the direct some of truncated submodules $W_1$, $W_2$ and trivial submodules $V_\gamma$ with $p<\gamma<q$. Note that $p,q$ are allowed to approach $\infty$.

\begin{case}
The complete $Z$-graded case
\end{case}
If for any $\gamma_0\in\Z$, either $f_{1,\gamma_0}$ or $f_{-1,\gamma_0+1}$ is a nonzero polynomial, we call $V$ a complete $\Z$-graded module. Note that this class of $B$-modules are indecomposable.

It will be proved later that the first class of $B$-modules must be trivial. Therefore we shall assume the modules discussed all belong to the second class if not declared.
\begin{lem}\label{lemm-F}
If for any $\gamma_0\in\Z$, not both $f_{1,\gamma_0}$ and $f_{-1,\gamma_0+1}$ are zero polynomials in \setcounter{case}{0} (\ref{use}). Then there must exist a constant $C\in\C$ and a set of polynomials $F_{\alpha,\gamma}(s)\in\C[s]$, such that $f_{\alpha,\gamma}(\lambda,\partial)=F_{\alpha,\gamma}((\alpha+\gamma+C)\lambda+\alpha\partial)$ for any $\alpha,\gamma\in\Z$. In particular, $f_{0,\gamma}(\lambda,\partial)=(\gamma+C)\lambda$ for all $\gamma\in\Z$.
\end{lem}
\begin{proof}
By setting $\alpha=\beta=0$ in equation (\ref{use}), we have
\begin{equation*}\label{use1}
f_{0,\gamma}(\mu,\partial+\lambda)f_{0,\gamma}(\lambda,\partial)=f_{0,\gamma}(\lambda,\partial+\mu)f_{0,\gamma}(\mu,\partial).
\end{equation*}
Comparing the the polynomial degrees of $\lambda$ on both sides, we get $\deg_y f_{0,\gamma}(x,y) + \deg_x f_{0,\gamma}(x,y)\\=\deg_x f_{0,\gamma}(x,y)$. Thus $\deg_y f_{0,\gamma}(x,y)=0$, which implies $f_{0,\gamma}(x,y)=h_\gamma(x)$ for some $h_\gamma(x)\in\C[x]$.

Let $\beta=0$ in equation (\ref{use}), we obtain
\begin{equation}\label{use2}
-\alpha\mu f_{\alpha,\gamma}(\lambda+\mu,\partial)=
h_\gamma(\mu)f_{\alpha,\gamma}(\lambda,\partial)-f_{\alpha,\gamma}(\lambda,\partial+\mu)h_{\alpha+\gamma}(\mu).
\end{equation}
Suppose $\deg_x f_{\alpha,\gamma}(x,y)=n$, then $\partial^n_\lambda f_{\alpha,\gamma}(\lambda,\partial)=\varphi_{\alpha,\gamma}(\partial)$ for some $\varphi_{\alpha,\gamma}(\partial)\in\C[\partial]$. Taking the $n$-th partial derivative with respect to $\lambda$ on both sides of equation (\ref{use2}), we obtain
\begin{equation*}\label{use3}
h_{\alpha+\gamma}(\mu)\varphi_{\alpha,\gamma}(\partial+\mu)=(h_\gamma(\mu)+\alpha\mu)\varphi_{\alpha,\gamma}(\partial).
\end{equation*}
Since $\varphi_{\alpha,\gamma}(\partial)$ dose not divide the left hand side of the equality, $\varphi_{\alpha,\gamma}(\partial)$ does not depend on $\partial$. By our supposition, either $\varphi_{1,\gamma}$ or $\varphi_{-1,\gamma+1}$ is not zero, which suggests $h_{\gamma+1}(\mu)=h_\gamma(\mu)+\mu$ for any $\gamma\in\Z$. Hence induction gives us $h_{\alpha+\gamma}(\mu)=h_\gamma(\mu)+\alpha\mu$ for all $\alpha,\gamma$. Now we have
\begin{equation*}
-\alpha\mu f_{\alpha,\gamma}(\lambda+\mu,\partial)=
(h_{\alpha+\gamma}(\mu)-\alpha\mu)f_{\alpha,\gamma}(\lambda,\partial)-f_{\alpha,\gamma}(\lambda,\partial+\mu)h_{\alpha+\gamma}(\mu),
\end{equation*}
that is
\begin{equation*}\label{use4}
\alpha\mu (f_{\alpha,\gamma}(\lambda+\mu,\partial)-f_{\alpha,\gamma}(\lambda,\partial))=
h_{\alpha+\gamma}(\mu)(f_{\alpha,\gamma}(\lambda,\partial+\mu)-f_{\alpha,\gamma}(\lambda,\partial)).
\end{equation*}

If $f_{\alpha,\gamma}(x,y)$ does not depend on $x$ and $y$, for any $\alpha,\gamma\in\Z,\alpha\neq 0$, equation  (\ref{use}) implies that $f_{\alpha,\gamma}$ is always a zero polynomial, which means $V$ is trivial. Thus there must be some $\alpha_0,\gamma_0\in\Z,\alpha_0\neq 0$, such that $f_{\alpha_0,\gamma_0}(x,y)$ does not depend on $x$ and $y$.

For this $f_{\alpha_0,\gamma_0}$, we have
\begin{equation*}
\alpha_0\mu (f_{\alpha_0,\gamma_0}(\lambda+\mu,\partial)-f_{\alpha_0,\gamma_0}(\lambda,\partial))=
h_{\alpha_0+\gamma_0}(\mu)(f_{\alpha_0,\gamma_0}(\lambda,\partial+\mu)-f_{\alpha_0,\gamma_0}(\lambda,\partial)).
\end{equation*}

We shall prove $h_{\alpha_0+\gamma_0}(\mu)=C_{\alpha_0+\gamma_0}\mu$ for some constant $C_{\alpha_0+\gamma_0}$.

Let $g_1(\lambda,\mu,\partial)=f_{\alpha_0,\gamma_0}(\lambda+\mu,\partial)-f_{\alpha_0,\gamma_0}(\lambda,\partial)$ and $g_2(\lambda,\mu,\partial)=f_{\alpha_0,\gamma_0}(\lambda,\partial+\mu)-f_{\alpha_0,\gamma_0}(\lambda,\partial)$. Then if $h_{\alpha_0+\gamma_0}(\mu)$ is not a zero polynomial, we can show $g_1$ and $g_2$ are not zero polynomials either. Suppose on the contrary $g_1=0$, then $g_2=0$. Thus $f_{\alpha_0,\gamma_0}(\lambda+\mu,\partial)=f_{\alpha_0,\gamma_0}(\lambda,\partial+\mu)=f_{\alpha_0,\gamma_0}(\lambda,\partial)$, which means $f_{\alpha_0,\gamma_0}$ is a constant polynomial and leads to a contradiction. Similarly, $g_1=0$ if $g_2=0$.

If $\mu=0$, we have $g_1=0$. Thus $\mu|g_1$. Since $\partial_\mu g_1|_{\mu=0}\neq 0$, $\mu^2\nmid g$. Since $g_1$ is not a zero polynomial for $\mu\neq 0$, the only non-constant polynomials in $\C[\mu]$ that divide $g_1$ are of the form $A\mu$ with $A\in\C$. Hence $g_1(\lambda,\mu,\partial)=\mu h_1(\lambda,\mu,\partial)$, where $h_1$ is not divisible by any polynomial in $\C[\mu]$. Similarly, $g_2=\mu h_2$ for some $h_2$ with the same property as $h_1$.

Now we have $\alpha_0\mu \cdot \mu h_1=
h_{\alpha_0+\gamma_0}(\mu)\mu h_2$, which implies $h_{\alpha_0+\gamma_0}(\mu)=C_{\alpha_0+\gamma_0}\mu$ for some constant $C_{\alpha_0+\gamma_0}\in\C$. Then $h_0(\mu)=h_{\alpha_0+\gamma_0}(\mu)-(\alpha_0+\gamma_0)\mu=(C_{\alpha_0+\gamma_0}-(\alpha_0+\gamma_0))\mu$.

From now on, we shall fix a constant $C\in\C$ such that $h_0(\mu)=C\mu$. Consequently $h_\alpha(\mu)=h_0(\mu)+\alpha\mu=(\alpha+C)\mu$ for any $\alpha\in\Z$. That is, $f_{0,\gamma}(x,y)=(\gamma+C)x$ for all $\gamma\in\Z$.

Now we have
\begin{equation}\label{use5}
\alpha \mu (f_{\alpha,\gamma}(\lambda+\mu,\partial)-f_{\alpha,\gamma}(\lambda,\partial))=
(\alpha+\gamma+C)\mu(f_{\alpha,\gamma}(\lambda,\partial+\mu)-f_{\alpha,\gamma}(\lambda,\partial)).
\end{equation}
If $\mu\neq 0$
\begin{equation*}
\alpha \cdot \frac{f_{\alpha,\gamma}(\lambda+\mu,\partial)-f_{\alpha,\gamma}(\lambda,\partial)}{\mu}=
(\alpha+\gamma+C)\frac{f_{\alpha,\gamma}(\lambda,\partial+\mu)-f_{\alpha,\gamma}(\lambda,\partial)}{\mu}.
\end{equation*}
Let $\mu$ approaches zero, we can obtain a set of partial differential equations (PDE) with respect to $f_{\alpha,\gamma}$ for any $\alpha,\gamma\in\Z$ in the following form

\begin{equation*}
\alpha \cdot \frac{\partial f_{\alpha,\gamma}(x,y)}{\partial x}=
(\alpha+\gamma+C)\frac{\partial f_{\alpha,\gamma}(x,y)}{\partial y},
\end{equation*}
where we should keep in mind that $f_{\alpha,\gamma}(x,y)$ is required to be a polynomial here.

We solve the above PDE under the initial condition $F(s)=f_{\alpha,\gamma}(0,s)$, where $F(s)$ is a polynomial in one indeterminate $s$.
Let $z=f_{\alpha,\gamma}(x,y)$, the ordinary differential equations of the integral curve are
\begin{equation*}
\frac{dx}{dt}=\alpha,\quad\frac{dy}{dt}=-(\alpha+\gamma+C),\quad\frac{dz}{dt}=0,
\end{equation*}
integration along the curve gives
\begin{equation*}
x=\alpha t+A,\quad y=-(\alpha+\gamma+C)t+B,\quad z=H,
\end{equation*}
where $A,B,H$ are constants to be determined.

The initial value of the PDE implies that the integral curve passes $(0,s,F(s))$, which means
\begin{equation*}
x=\alpha t,\quad y=-(\alpha+\gamma+C)t+s,\quad z=F(s).
\end{equation*}
We assume $\alpha\neq 0$, the first two equations will give us
\begin{equation*}
s=y+\frac{\alpha+\gamma+C}{\alpha}\cdot x.
\end{equation*}
Thus
\begin{equation*}
z=F(y+\frac{\alpha+\gamma+C}{\alpha}\cdot x),
\end{equation*}
which means the general solution for the PDE is
\begin{equation}\label{solution}
f_{\alpha,\gamma}(\lambda,\partial)=F_{\alpha,\gamma}((\alpha+\gamma+C)\lambda+\alpha\partial),
\end{equation}
for some polynomial $F_{\alpha,\gamma}(s)\in\C[s]$. Note if $\alpha=0$, $f_{0,\gamma}(\lambda,\partial)=(\gamma+C)\lambda$ can be fitted into equation (\ref{solution}). Thus equation (\ref{solution}) holds for any $\alpha,\gamma\in\Z$.
\end{proof}

A similar result can be set up for the first case of intermediate series $B$-modules:

\begin{lem}\label{pre6}
Let $V$ be an intermediate series module of the first class, that is, $V=W_1\bigoplus W_2\bigoplus_{p<\gamma<q}V_\gamma$ as a $B$-module. Then there must be some constants $C_1,C_2\in\C$ and a set of polynomials $F_{\alpha,\gamma}(s)\in\C[s]$, such that
\[f_{\alpha,\gamma}(\lambda,\partial)=\left\{\begin{array}{ll}
F_{\alpha,\gamma}((\alpha+\gamma+C_1)\lambda+\alpha\partial)&\mbox{if $\alpha+\gamma\leq p$ and $\gamma\leq p$},\\[4pt]
F_{\alpha,\gamma}((\alpha+\gamma+C_2)\lambda+\alpha\partial)&\mbox{if $\alpha+\gamma\geq q$ and $\gamma\geq q$},\\[4pt]
0& else where.
\end{array} \right.
\]
\end{lem}

Now equation (\ref{use}) can be written as
\begin{align}\label{use+}
\begin{split}
&(\beta\lambda-\alpha\mu)F_{\alpha+\beta,\gamma}((\alpha+\beta+\gamma+C)(\lambda+\mu)+(\alpha+\beta)\partial)\\
=&F_{\beta,\gamma}((\beta+\gamma+C)\mu+\beta(\partial+\lambda))F_{\alpha,\beta+\gamma}((\alpha+\beta+\gamma+C)\lambda+\alpha\partial)\\
&-F_{\alpha,\gamma}((\alpha+\gamma+C)\lambda+\alpha(\partial+\mu))F_{\beta,\alpha+\gamma}((\alpha+\beta+\gamma+C)\mu+\beta\partial).
\end{split}
\end{align}

For convenience, let's denote $\deg F_{\alpha,\gamma}(s)=n_{\alpha,\gamma}$.

We are finally able to get rid of the annoying zero structure coefficients:

\begin{lem}\label{nonzero}
If $F_{\alpha_0,\gamma_0}=0$ for some $\alpha_0,\gamma_0\in\Z$, where $\alpha_0\neq 0$, then $F_{\alpha,\gamma}= 0$ for all $\alpha,\gamma\in\Z$, that is, $V$ must be a trivial $B$-module.
\end{lem}
\begin{proof}
If $F_{\alpha_0,\gamma_0}=0$ for some $\alpha_0,\gamma_0\in\Z$, where $\alpha_0\neq 0$, let $\alpha=\alpha_0,\gamma=\gamma_0$ in equation (\ref{use+}). We have
\begin{align*}
\begin{split}
&(\beta\lambda-\alpha_0\mu)F_{\alpha_0+\beta,\gamma_0}((\alpha_0+\beta+\gamma_0+C)(\lambda+\mu)+(\alpha_0+\beta)\partial)\\
=&F_{\beta,\gamma_0}((\beta+\gamma_0+C)\mu+\beta(\partial+\lambda))F_{\alpha_0,\beta+\gamma_0}((\alpha_0+\beta+\gamma_0+C)\lambda+\alpha_0\partial).
\end{split}
\end{align*}
If $n_{\beta,\gamma_0}\geq 1$ or $n_{\alpha_0,\beta+\gamma_0}\geq 1$, the expansion on the right hand side of the equality will contain a term only have one variable $\partial$, but the left hand side does not. Thus $n_{\beta,\gamma_0}=n_{\alpha_0,\beta+\gamma_0}=0$, which implies $F_{\alpha_0+\beta,\gamma_0}= 0$ for any $\beta\in\Z$. Therefore we have proved if $F_{\alpha_0,\gamma_0}= 0$, then $F_{\alpha,\gamma_0}=0$ for all $\alpha\in\Z$. In particular, we get $F_{1,\gamma_0}=F_{-1,\gamma_0}= 0$. By Lemma \ref{pre4}, for any $\gamma\in\Z$, we can obtain a $\beta_0\neq 0$ such that $F_{\beta_0,\gamma}= 0$. Hence $F_{\alpha,\gamma}= 0$ for any $\alpha,\gamma\in\Z$.
\end{proof}

By Lemma \ref{pre6} and the proof of Lemma \ref{nonzero}, we can conclude:
\begin{pro}
An intermediate series $B$-module belonging to the first case must be trivial.
\end{pro}

Thus we need only to focus on classifying the $B$-modules in the second class.

\begin{pro}\label{almost1}
Let $V=\bigoplus_{\gamma\in\Z}\C[\partial]v_\gamma$ be a nontrivial free intermediate $B$-module with $\C[\partial]$-basis $\{v_\gamma\,|\,\gamma\in\Z\}$ and structure coefficients $\{f_{\alpha,\gamma}(\lambda,\partial)=F_{\alpha,\gamma}((\alpha+\gamma+C)\lambda+\alpha\partial)\,|\,\alpha,\gamma\in\Z\}$. If $C\notin\Z$, then $V$ must be isomorphic to some $V_{C,D}$ with $C,D\in\C$ which defined in Section 1.
\end{pro}

\begin{proof}
Let $\beta= 0$ in equation (\ref{use+}), we get
\begin{eqnarray*}
&&\alpha F_{\alpha,\gamma}((\alpha+\gamma+C)(\lambda+\mu)+\alpha\partial)\\
&=&\!\!\!
(\alpha+\gamma+C)F_{\alpha,\gamma}((\alpha+\gamma+C)\lambda+\alpha(\partial+\mu))-(\gamma+C)F_{\alpha,\gamma}((\alpha+\gamma+C)\lambda+\alpha\partial).
\end{eqnarray*}
If $n_{\alpha,\gamma}\neq 0$, comparing the coefficients of $\mu^{n_{\alpha,\gamma}}$, we obtain
\begin{equation}\label{eq-degree}
\alpha(\alpha+\gamma+C)^{n_{\alpha,\gamma}}=(\alpha+\gamma+C)\alpha^{n_{\alpha,\gamma}},
\end{equation}
which implies $n_{\alpha,\gamma}=1$ for $\alpha\neq 0$, since $C\notin\Z$. Because $F_{0,\gamma}((\gamma+C)\lambda)=f_{0,\gamma}(\lambda,\partial)=(\gamma+C)\lambda$, we can assume $F_{0,\gamma}(x)=x$, i.e, $n_{\alpha,\gamma}=1$ for any $\gamma\in\Z$.

Since for any $\alpha,\gamma\in\mathbb{Z}$, $n_{\alpha,\gamma}\leq 1$, we can suppose $F_{\alpha,\gamma}((\alpha+\gamma+C)\lambda+\alpha\partial)=P_{\alpha,\gamma}((\alpha+\gamma+C)\lambda+\alpha\partial)+R_{\alpha,\gamma}$, where $P_{\alpha,\gamma},R_{\alpha,\gamma}\in\C$. We already have $P_{0,\gamma}=1$ and $R_{0,\gamma}=0$ for any $\gamma\in\Z$.

Now equation (\ref{use+}) becomes
\begin{align}\label{sleepy}
\begin{split}
&(\beta\lambda-\alpha\mu)(P_{\alpha+\beta,\gamma}((\alpha+\beta+\gamma+C)(\lambda+\mu)+(\alpha+\beta)\partial)+R_{\alpha+\beta,\gamma})\\
=&(P_{\beta,\gamma}((\beta+\gamma+C)\mu+\beta(\partial+\lambda))+R_{\beta,\gamma})(P_{\alpha,\beta+\gamma}((\alpha+\beta+\gamma+C)\lambda+\alpha\partial)+R_{\alpha,\beta+\gamma})\\
&-(P_{\alpha,\gamma}((\alpha+\gamma+C)\lambda+\alpha(\partial+\mu))+R_{\alpha,\gamma})(P_{\beta,\alpha+\gamma}((\alpha+\beta+\gamma+C)\mu+\beta\partial)+R_{\beta,\alpha+\gamma}).
\end{split}
\end{align}

Since $C\notin\Z$, $\alpha+\beta+\gamma+C\neq 0$.  Comparing the coefficients of $\lambda^2$ in equation (\ref{sleepy}), we have $\beta(\alpha+\beta+\gamma+C) P_{\alpha+\beta,\gamma}=\beta(\alpha+\beta+\gamma+C) P_{\alpha,\beta+\gamma}P_{\beta,\gamma}$, which gives us $P_{\alpha+\beta,\gamma}=P_{\alpha,\beta+\gamma}P_{\beta,\gamma}$ for $\beta\neq 0$. Since $P_{0,\gamma}=1$, $P_{\alpha+\beta,\gamma}=P_{\alpha,\beta+\gamma}P_{\beta,\gamma}$ also holds for $\beta=0$. Thus we always have
\begin{equation}\label{work1}
P_{\alpha+\beta,\gamma}=P_{\alpha,\beta+\gamma}P_{\beta,\gamma}.
\end{equation}



For any $\alpha\geq 1$, applying (\ref{work1}) repeatedly, we can obtain
\begin{align*}
P_{\alpha,\gamma}&=P_{1,\gamma}P_{\alpha-1,\gamma+1}\\
                  &=P_{1,\gamma}P_{1,\gamma+1}P_{\alpha-2,\gamma+2}\\
                  &=\cdots \cdots\\
                  &=P_{1,\gamma}P_{1,\gamma+1}\cdots P_{2,\alpha+\gamma-2}\\
                  &=P_{1,\gamma}P_{1,\gamma+1}\cdots P_{1,\alpha+\gamma-2}P_{1,\alpha+\gamma-1},
\end{align*}
that is
\begin{equation}\label{work2}
P_{\alpha,\gamma}=\prod\limits_{i=0}^{\alpha-1}P_{1,\gamma+i}.
\end{equation}
Similarly, if $\alpha\leq -1$, we have
\begin{equation}\label{work3}
P_{\alpha,\gamma}\prod\limits_{i=\alpha}^{-1}P_{1,\gamma+i}=P_{0,\alpha+\gamma}.
\end{equation}
Since $P_{\alpha,\gamma}$ are not equal to zero, we can choose a proper family of constants $\Gamma=\{C_\alpha\in\C|C_\alpha\neq 0,\alpha\in\Z\}$ such that $P_{1,\gamma}=\frac{C_{\gamma+1}}{C_\gamma}$. Let $\alpha=1$ and $\beta=0$ in equation (\ref{work1}), we get $P_{1,\gamma}=P_{0,\gamma}P_{1,\gamma}$. Thus $P_{0,\gamma}=1$ for any $\gamma$. Now equation (\ref{work2}) and (\ref{work3}) turn into $P_{\alpha,\gamma}=\prod\limits_{i=0}^{\alpha-1}\frac{C_{\gamma+i+1}}{C_{\gamma+i}}$ if $\alpha\geq 1$ and $P_{\alpha,\gamma}\prod\limits_{i=\alpha}^{-1}\frac{C_{\gamma+i+1}}{C_{\gamma+i}}=1$ if $\alpha\leq -1$, which suggest $P_{\alpha,\gamma}=\frac{C_{\alpha+\gamma}}{C_\gamma}$ for any $\alpha\in\Z$.

Now let $R_{\alpha,\gamma}=\frac{C_{\alpha+\gamma}}{C_\gamma}\cdot\alpha K_{\alpha,\gamma}$ for some $K_{\alpha,\gamma}\in\C$, we have
\begin{equation*}
F_{\alpha,\gamma}((\alpha+\gamma+C)\lambda+\alpha\partial)=\frac{C_{\alpha+\gamma}}{C_\gamma}((\alpha+\gamma+C)\lambda+\alpha\partial+\alpha K_{\alpha,\gamma}).
\end{equation*}

Let's change the $\C[\partial]$-basis of $V$ $v_\gamma$ to $v'_\gamma=C_\gamma v_\gamma$. We have
\begin{equation*}
F_{\alpha,\gamma}((\alpha+\gamma+C)\lambda+\alpha\partial)=(\alpha+\gamma+C)\lambda+\alpha\partial+\alpha K_{\alpha,\gamma}.
\end{equation*}

Consequently equation (\ref{sleepy}) becomes
\begin{align}\label{use10}
\begin{split}
&(\beta\lambda-\alpha\mu)((\alpha+\beta+\gamma+C)(\lambda+\mu)+(\alpha+\beta)\partial+(\alpha+\beta)K_{\alpha+\beta,\gamma})\\
=&((\beta+\gamma+C)\mu+\beta(\partial+\lambda)+\beta K_{\beta,\gamma})((\alpha+\beta+\gamma+C)\lambda+\alpha\partial+\alpha K_{\alpha,\beta+\gamma})\\
&-((\alpha+\gamma+C)\lambda+\alpha(\partial+\mu)+\alpha K_{\alpha,\gamma})((\alpha+\beta+\gamma+C)\mu+\beta\partial+\beta K_{\beta,\alpha+\gamma}).
\end{split}
\end{align}

Comparing the coefficients of $\lambda$ and $\partial$ in  equation (\ref{use10}), we get
\begin{equation*}
(\alpha+\beta) K_{\alpha+\beta,\gamma}=\alpha K_{\alpha,\beta+\gamma}+(\alpha+\beta+\gamma+C)K_{\beta,\gamma}-(\alpha+\gamma+C)K_{\beta,\alpha+\gamma},
\end{equation*}
\begin{equation}\label{labour1}
K_{\alpha,\beta+\gamma}+K_{\beta,\gamma}=K_{\beta,\alpha+\gamma}+K_{\alpha,\gamma}.
\end{equation}

Let $\alpha=\beta=1$ in the first equation of \eqref{labour1}, we have
\begin{equation}\label{eq-k1}
2K_{2,\gamma}=K_{1,1+\gamma}+(2+\gamma+C)K_{1,\gamma}-(1+\gamma+C)K_{1,1+\gamma}
=(2+\gamma+C)K_{1,\gamma}-(\gamma+C)K_{1,1+\gamma}.
\end{equation}
Let $\alpha=2,\beta=1$ in the second equations of \eqref{labour1}, we have
\begin{eqnarray}\label{eq-k2}
K_{2,1+\gamma}+K_{1,\gamma}=K_{1,\gamma+2}+K_{2,\gamma}.
\end{eqnarray}
By equations \eqref{eq-k1} and \eqref{eq-k2}, we have $K_{1,\gamma}=D_1$ for some constant $D_1\in\C$. By the second equation in \eqref{labour1}, there is $K_{\alpha,\gamma}=D_{\alpha}$ for some constant $D_{\alpha}\in\C$. By the first equation in \eqref{labour1}, we have $D_{\alpha}=D$ for all $\alpha\neq0$.

Note that we can suppose $K_{0,\gamma}=D$ without affecting the expression of $F_{0,\gamma}$, therefore we get $K_{\alpha,\beta}=D$ for all $\alpha,\beta\in\Z$. Thus we have
\begin{equation*}
F_{\alpha,\gamma}((\alpha+\gamma+C)\lambda+\alpha\partial)=(\alpha+\gamma+C)\lambda+\alpha(\partial+D)
\end{equation*}
for all $\alpha,\gamma\in\Z$.
This means $L_\alpha(w)_\lambda v'_\gamma=((\alpha+\gamma+C)\lambda+\alpha(\partial+D))v'_{\alpha+\gamma}$, that is, $V$ is isomorphic to $V_{C,D}$.
\end{proof}

\begin{pro}\label{almost2}
Let $V=\bigoplus_{\gamma\in\Z}\C[\partial]v_\gamma$ be a nontrivial free intermediate $B$-module with $\C[\partial]$-basis $\{v_\gamma\,|\,\gamma\in\Z\}$ and structure coefficients $\{f_{\alpha,\gamma}=F_{\alpha,\gamma}\,|\,\alpha,\gamma\in\Z\}$. If $C\in\Z$, then $V$ must be isomorphic to $V_{C,D}$, $V_D$ or $V'_D$ for some $C,D\in\C$, which defined in Section 1.
\end{pro}
\begin{proof}


For convenience, we change the $\C[\partial]$-basis $\{v_\gamma\,|\,\gamma\in\Z\}$ into $\{v'_\gamma=v_{\gamma-C}\,|\,\gamma\in\Z\}$. Then the structure coefficients associated with $\{v'_\gamma\,|\,\gamma\in\Z\}$ are $\{f_{\alpha,\gamma}(\lambda,\partial)=F_{\alpha,\gamma}((\alpha+\gamma)\lambda+\alpha\partial)\,|\,\alpha,\gamma\in\Z\}$, then equation (\ref{use+}) becomes
\begin{align}\label{use++}
\begin{split}
&(\beta\lambda-\alpha\mu)F_{\alpha+\beta,\gamma}((\alpha+\beta+\gamma)(\lambda+\mu)+(\alpha+\beta)\partial)\\
=&F_{\beta,\gamma}((\beta+\gamma)\mu+\beta(\partial+\lambda))F_{\alpha,\beta+\gamma}((\alpha+\beta+\gamma)\lambda+\alpha\partial)\\
&-F_{\alpha,\gamma}((\alpha+\gamma)\lambda+\alpha(\partial+\mu))F_{\beta,\alpha+\gamma}((\alpha+\beta+\gamma)\mu+\beta\partial).
\end{split}
\end{align}
By equation \eqref{eq-degree}, it is only possible that $n_{\alpha,\gamma}\geq 2$ for $\gamma= 0$ or $\alpha+\gamma=0$.

If $n_{\alpha,\gamma}< 2$ for all $\alpha,\gamma\in\Z$,  by the similar proof of Proposition \ref{almost1}, we also have the module $V$ is isomorphism to $V_{C,D}$.

If $n_{\beta_0,0}\geq 2$ for some $\beta_0\neq 0$. Let $\alpha+\gamma=0$ in equation (\ref{use++}), we get
\begin{align*}
\begin{split}
&(\beta\lambda-\alpha\mu)F_{\alpha+\beta,\gamma}(\beta(\lambda+\mu)+(\alpha+\beta)\partial)\\
=&F_{\beta,\gamma}((\beta+\gamma)\mu+\beta(\partial+\lambda))F_{\alpha,\beta+\gamma}(\beta\lambda+\alpha\partial)\\
&-F_{\alpha,\gamma}(\alpha(\partial+\mu))F_{\beta,0}(\beta(\mu+\partial)),
\end{split}
\end{align*}
which means
\begin{equation}\label{whatnow}
n_{\alpha,\gamma}+n_{\beta,0}\leq max\{n_{\alpha+\beta,\gamma}+1,\ n_{\beta,\gamma}+n_{\alpha,\beta+\gamma}\}.
\end{equation}

If $\alpha,\gamma\neq 0$ and $\alpha+\gamma=0$, $\beta_0+\gamma\neq 0$, we have $n_{\alpha+\beta_0,\gamma},n_{\beta_0,\gamma},n_{\alpha,\beta_0+\gamma}\leq 1$, which implies
\begin{equation*}
n_{\alpha,\gamma}+n_{\beta_0,0}\leq max\{n_{\alpha+\beta_0,\gamma}+1,\ n_{\beta_0,\gamma}+n_{\alpha,\beta_0+\gamma}\}\leq 2.
\end{equation*}
Thus we must have $n_{\beta_0,0}=2$ and $n_{\alpha,\gamma}=0$ for any $\alpha+\gamma=0$, $\alpha\neq 0$ and $\beta_0+\gamma\neq 0$.

On the other hand, if $\alpha+\gamma=0$, $\beta_0+\gamma= 0$. We have $\alpha=\beta_0=-\gamma$. Let $\alpha=\beta_0,\beta=\beta_0,\gamma=-\beta_0$ in equation (\ref{use++}), we obtain
\begin{align}\label{what}
\begin{split}
&\beta_0(\lambda-\mu)F_{2\beta_0,-\beta_0}(\beta_0(\lambda+\mu)+2\beta_0\partial)\\
=&F_{\beta_0,-\beta_0}(\beta_0(\partial+\lambda))F_{\beta_0,0}(\beta_0(\lambda+\partial))
-F_{\beta_0,-\beta_0}(\beta_0(\partial+\mu))F_{\beta_0,0}(\beta_0(\mu+\partial)).
\end{split}
\end{align}
Suppose $n_{\beta_0,-\beta_0}\geq 1$. Notice that $n_{2\beta_0,-\beta_0}\leq 1$. The polynomial degree for $\lambda$ on the right side of equation (\ref{what}) will be larger than 3 but smaller than 2 on the left hand side. Thus we must have $n_{\beta_0,-\beta_0}=0$.

Now we have proved that $n_{\alpha,\gamma}=0$ holds for any $\alpha+\gamma=0$, $\alpha\neq 0$. Hence if $\alpha+\gamma=0$, expression (\ref{whatnow}) becomes
\begin{equation*}
n_{\beta,0}\leq max\{n_{\alpha+\beta,\gamma}+1,\ n_{\beta,\gamma}+n_{\alpha,\beta+\gamma}\}.
\end{equation*}
For any $\beta\neq 0$, we can choose a proper pair of $\alpha$ and $\gamma$ so that $\gamma\neq 0$, $\gamma+\beta\neq 0$ and $\alpha+\beta+\gamma\neq 0$, which in turn gives us $n_{\alpha+\beta,\gamma},n_{\beta,\gamma},n_{\alpha,\beta+\gamma}\leq 1$. Hence for any $\beta\neq 0$, we have obtained
\begin{equation*}
n_{\beta,0}\leq max\{n_{\alpha+\beta,\gamma}+1,\ n_{\beta,\gamma}+n_{\alpha,\beta+\gamma}\}\leq 2
\end{equation*}

Now we have proved that $n_{\alpha,\gamma}=0$ for all $\alpha+\gamma=0$, $\alpha\neq 0$, and $n_{\beta,0}\leq 2$ for all $\beta\neq 0$. Consequently, $F_{\alpha,\gamma}$ can be written as
\begin{equation*}
F_{\alpha,\gamma}((\alpha+\gamma)\lambda+\alpha\partial)=\left\{ \begin{array}{ll}
P_{\alpha,\gamma}((\alpha+\gamma)\lambda+\alpha\partial)+R_{\alpha,\gamma}& \mbox{if $\gamma\neq 0$ and $\alpha+\gamma\neq 0$,}\\[4pt]
\alpha (P_{\alpha,0}(\lambda+\partial)^2+Q_{\alpha,0}(\lambda+\partial)+ R_{\alpha,0})& \mbox{if $\gamma= 0$,}\\[4pt]
\alpha P_{\alpha,\gamma}& \mbox{if $\alpha+\gamma= 0$,}
 \end{array}  \right.
\end{equation*}
where $P_{\alpha,\gamma},Q_{\alpha,\gamma},R_{\alpha,\gamma}\in\C$ are coefficients. The coefficient $\alpha$ added when $\gamma= 0$ and $\alpha+\gamma= 0$ is meant for the convenience of later steps.

As in the proof of Proposition \ref{almost1}, assuming $P_{0,0}=1$, we still have $P_{\alpha,\gamma}=\frac{C_{\alpha+\gamma}}{C_\gamma}$ for a given set of constants $\{C_\alpha\in\C\,|\,C_\alpha\neq 0,\alpha\in\Z\}$. Now use the technique to rescale the $\C[\partial]$-basis of $V$ from $v_\gamma$ to $v'_\gamma=C_\gamma v_\gamma$ again, we obtain a refined form of $F_{\alpha,\gamma}$ as follows
\begin{equation*}
F_{\alpha,\gamma}((\alpha+\gamma)\lambda+\alpha\partial)=\left\{ \begin{array}{ll}
(\alpha+\gamma)\lambda+\alpha\partial+\alpha K_{\alpha,\gamma}& \mbox{if $\gamma\neq 0$ and $\alpha+\gamma\neq 0$,}\\[4pt]
\alpha ((\lambda+\partial)^2+H_{\alpha,0}(\lambda+\partial)+ T_{\alpha,0})& \mbox{if $\gamma= 0$,}\\[4pt]
\alpha & \mbox{if $\alpha+\gamma= 0$,}
 \end{array}  \right.
\end{equation*}
where $K_{\alpha,\gamma},H_{\alpha,0},T_{\alpha,0}\in\C$ are coefficients to be specified.

In equation (\ref{use++}), let $\alpha=1,\beta=1,\gamma=0$, we have
\begin{equation*}
(\lambda-\mu)(\lambda+\mu+2\partial+K_{2,-1})=(\lambda+\partial)^2+H_{1,0}(\lambda+\partial)+ T_{1,0}-((\mu+\partial)^2+H_{1,0}(\mu+\partial)+ T_{1,0}).
\end{equation*}
Thus $H_{1,0}=2K_{2,-1}$.

Similarly, if $\alpha=1,\beta=1,\gamma=-1$, we have $H_{1,0}=H_{2,0}$ and $T_{1,0}=T_{2,0}$.

If $\alpha=1,\beta=2,\gamma=-1$, we get $H_{2,0}=K_{1,1}+K_{2,-1}$ and $T_{2,0}=K_{1,1}K_{2,-1}$.

If $\alpha=1,\beta=2,\gamma=-2$, we obtain $H_{1,0}=K_{1,-2}+K_{2,-1}$.

Suppose $K_{1,1}=D$ for some constant $D$. Then from the relations above, we can deduce that $K_{1,-2}=K_{2,-1}=K_{1,1}=D$, $H_{1,0}=H_{2,0}=2D$ and $T_{1,0}=T_{2,0}=D^2$.

Thus by equation (\ref{use++}) and using induction, we have
\begin{equation*}
F_{\alpha,\gamma}((\alpha+\gamma)\lambda+\alpha\partial)=\left\{ \begin{array}{ll}
(\alpha+\gamma)\lambda+\alpha(\partial+D)& \mbox{if $\gamma\neq 0$ and $\alpha+\gamma\neq 0$},\\[4pt]
\alpha (\lambda+\partial+D)^2& \mbox{if $\gamma= 0$},\\[4pt]
\alpha & \mbox{if $\alpha+\gamma= 0$,}
 \end{array}  \right.
\end{equation*}
for all $\alpha\in\Z^+,\gamma\in\Z$.
By symmetry we can show that the above formula also holds for all $\alpha,\gamma\in\Z$.
Thus $V$ is isomorphic to $V_D$ as a $B$-module.

On the other hand, if $n_{\alpha_0,\gamma_0}\geq 2$ for some $\alpha_0,\gamma_0\neq 0$ such that $\alpha_0+\gamma_0=0$. Similar treatment as above can show us
\[
F_{\alpha,\gamma}((\alpha+\gamma)\lambda+\alpha\partial)=\left\{ \begin{array}{ll}
(\alpha+\gamma)\lambda+\alpha(\partial+D)& \mbox{if $\gamma\neq 0$ and $\alpha+\gamma\neq 0$,}\\[4pt]
\alpha & \mbox{if $\gamma= 0$,}\\[4pt]
\alpha (\partial+D)^2& \mbox{if $\alpha+\gamma= 0$,}
 \end{array}  \right.
\]
for some constant $D\in\C$, which means $V$ is isomorphic to $V'_D$.

Thus we have proved that $V$ must be isomorphic to either $V_D$ or $V'_D$ for some $D\in\C$.
\end{proof}

Summarizing Proposition \ref{almost1} and \ref{almost2}, we conclude that Theorem \ref{thm-mod} holds.

\end{document}